\documentclass[a4paper,12pt]{amsart}
\usepackage{amssymb}
\usepackage{amsmath}
\usepackage{pgf,tikz}
\usepackage{mathrsfs}
\usepackage{pgf,xcolor}
\usetikzlibrary{arrows}
\usepackage{a4wide}
 \DeclareMathOperator{\tr}{tr}

\DeclareMathOperator{\Img}{Im}

\DeclareMathOperator{\Hom}{Hom}
\DeclareMathOperator{\U}{U}

\newcommand{\CP}{{\mathbb C}P}
\newcommand{\C}{{\mathbb C}}

\newtheorem{theorem}{Theorem}[section]
\newtheorem{corollary}[theorem]{Corollary}
\newtheorem{lemma}[theorem]{Lemma}

\newtheorem{definition}[theorem]{Definition}
\newtheorem{proposition/definition}[theorem]{Proposition/Definition}

\theoremstyle{definition}
\newtheorem{eg}[theorem]{Example}
\newtheorem{remark}[theorem]{Remark}
\newtheorem{example}[theorem]{Example}

\usepackage{pdfpages}

\usepackage{hyperref}

\title[Primitive immersions of constant curvature]
{Primitive  immersions of constant curvature of surfaces into  flag manifolds}
\author[R. Pacheco]{Rui Pacheco}
\address{Centro de Matem\'{a}tica e Aplica\c{c}{\~{o}}es (CMA-UBI), Universidade da Beira Interior, 6201 -- 001
	Covilh{\~{a}}, Portugal.}
\email{rpacheco@ubi.pt, mehmood.ur.rehman@ubi.pt}

\author[M. U. Rehman]{Mehmood Ur Rehman}

\thanks{The first author was partially supported by Funda\c{c}\~{a}o para a Ci\^{e}ncia e Tecnologia through the project UID/MAT/00212/2019. The second author was partially supported by Funda\c{c}\~{a}o para a Ci\^{e}ncia e Tecnologia through the grant UI/BD/153058/2022.}


\begin{document}

\maketitle

\begin{abstract}
We investigate certain immersions of constant curvature from Riemann surfaces into flag manifolds equipped  with  invariant metrics, namely primitive lifts associated to pseudoholomorphic maps of surfaces into complex Grassmannians. We prove that a  primitive immersion from the two-sphere into the full flag manifold which has constant curvature with respect to \emph{at least one}  invariant metric is unitarily equivalent to the primitive lift of a Veronese map, hence it has constant curvature with respect to \emph{all} invariant metrics. 
We prove a partial generalization of this result to the case where the domain is a general simply connected Riemann surface. On the way,  we consider the problem of finding the  invariant metric on the flag manifold, under a certain normalization condition,  that maximizes the induced area of the two-sphere by a given primitive immersion.
	\end{abstract}

\section{Introduction}   E.~Calabi  in his PhD thesis considered  the  problem of finding which complex Hermitian manifolds can be holomorphically and isometrically embedded into a complex space form (see \cite{LawsonCalabi} and the references therein).  His answers to this problem led to important applications, among which we can  find the local classification of complex submanifolds with constant holomorphic sectional curvature of a complex space form.  In particular, E. Calabi showed that any such submanifold of the complex projective space $\mathbb{C}P^{n-1}$ is locally unitarily congruent to a piece of a \emph{round Veronese embedding.}  Making use of this classification result, J. Bolton et al. \cite{Bolton} proved that any conformal minimal immersion of constant curvature of the  two-sphere $S^2$ in the complex projetive space  belongs to the \emph{Veronese sequence}, up to unitary congruence. This was later generalized by Q.-S. Chi and Y. Zheng to the case where the domain is a simply connected surface, not necessarily closed, in \cite[Theorem 2]{Chi-Zheng}.

For isometric immersions, minimality  is equivalent to harmonicity  \cite{eells-sampson,urakawa}, hence the methods of  harmonic map  theory can be brought to the study of minimal immersions of constant curvature. We are particularly interested in the following two methods:

\begin{enumerate}
\item \emph{Harmonic sequences.} The study of harmonic maps from Riemann surfaces into complex Grassmanians through their harmonic sequences was initiated by J.~G.~Wolfson \cite{Wolfson} and further developed by J.~C.~Wood and his collaborators in a sequence of papers \cite{bahy-wood-G2, bahy-wood-HPn,burstall-wood,erdem-wood}.   
 More recently, the method of harmonic sequences  has been intensively used   to classify minimal immersions of constant curvature of $S^2$  into different Riemannian symmetric spaces  \cite{HeJiao2014,HeJiao,Jiao2008,JiaoLi Qn, JiaoLi Q5,JiaoPeng2003,JiaoXu,JXX, LiHeJiaoQ3,PengWangXu}.

\item \emph{Twistorial constructions.} On the other hand, there exists a well-established  theory on  twistorial constructions of harmonic maps  from Riemann surfaces into symmetric spaces  \cite{burstall,burstallrawnsley}.  An important class of twistor lifts is that of \emph{primitive maps} into $k$-symmetric spaces \cite{burstall,Guest}. 
\end{enumerate} 

A natural problem arising from the above observations is that of characterizing the  primitive immersions with  constant curvature of a Riemann surface  into  a $k$-symmetric space $G/K$ (equipped with  $G$-invariant metrics). In the present paper we address this question in the important case of  {primitive lifts} of \emph{pseudoholomorphic}  maps \cite{JiaoPeng2003} from $S^2$ into complex Grassmannians. 
Recall that, by definition,  any pseudoholomorphic map into a complex Grassmannian $G_r(\C^n)$ belongs to the harmonic sequence $\psi_0,\ldots,\psi_p$ associated to some holomorphic map $\psi_0$ (which we will  assume to be linearly  full) into a complex Grassmannian $G_{r_0}(\C^n)$, with $r_0\geq r$; the corresponding \emph{primitive  lift} $\Psi=(\psi_0,\ldots,\psi_p)$ takes values in some flag manifold of the form
$$\frac{\U(n)}{\U(r_0)\times\ldots \times \U(r_p)},$$
with $r_0+\ldots+r_p=n$, endowed with its natural structure of $k$-symmetric space, with $k=p+1$. Note that the complex Grassmannian admits only one (up to positive scalar multiplication) $\U(n)$-invariant metric, while the flag manifold (for $p>1$) admits infinitely many  nonequivalent $\U(n)$-invariant metrics. In Theorem \ref{principaltheorem} we will  prove that \emph{if any such primitive lift $\Psi=(\psi_0,\ldots,\psi_p)$ from $S^2$ has constant curvature with respect to at least one invariant metric, then it has constant curvature with respect to all invariant metrics; moreover, each $\psi_j:S^2\to G_{r_j}(\C^n)$ is an immersion of constant curvature with constant K\"{a}hler angle.} Since all harmonic maps from $S^2$ into $\CP^{n-1}$ are pseudoholomorphic \cite{eellswood}, we will conclude  (see Corollary  \ref{principalcorollary}) that \emph{any full primitive immersion from $S^2$ into the full flag manifold which has constant curvature with respect to at least one  invariant metric is unitarily equivalent to the primitive lift of a Veronese map.} In the final section, we prove a partial generalization of this result to the case where the domain is a general simply connected Riemann surface (not necessarily closed). The technique introduced by Q.-S. Chi and Y. Zheng \cite{Chi-Zheng} will play an important role in this generalization. 
On the way,  we consider the problem of finding the  invariant metric on the flag manifold, under a certain normalization condition,  that maximizes the induced area of $S^2$ by a given primitive immersion. 

\subsection*{Aknowledgements} We would like to thank to Zhenxiao Xie from Beihang University for his suggestions and interest on this topic, and to John C. Wood for sharing with us his notes on the proof of \cite[Theorem 2]{Chi-Zheng}.

\section{Preliminaries}

	\subsection{Harmonic sequences}We start by recalling from \cite{Bolton,burstall-wood,erdem-wood,JiaoPeng2003,Wolfson} the definition and properties of harmonic sequences associated to  harmonic maps from Riemann surfaces into complex Grassmannians. 
	
	We consider on $\C^{n}$ the standard Hermitian inner product  
$$\left<  v, w \right>=v_1\overline{w}_1+\ldots+v_n\overline{w}_n,$$
for $v=(v_1,\ldots,v_n),w=(w_1,\ldots,w_n)\in\C^n$. 
The Grassmannian $G_k(\C^{n})$ of all $k$-dimensional complex subspaces of $\C^{n}$ is a Hermitian symmetric $\U(n)$-space,  with stabilizers conjugate to $\U(k)\times \U(n-k)$.  Given $L\in G_k(\C^{n})$, the complex structure at the corresponding tangent space is given by \cite[\S 0.(B)]{burstall-wood}
\begin{equation}\label{eq:complexstructure}
	T_L^\C G_k(\C^n)=T_L^{1,0}G_k(\C^n)\oplus T_L^{0,1}G_k(\C^n) \cong \Hom(L,L^\perp)\oplus \Hom(L^\perp,L),
\end{equation}
while the compatible Riemann metric $h=\mathrm{Re}\, h_\C$, where $h_\C$ is the Hermitian metric, is given by
\begin{equation}\label{hmetric}
	h(\xi,\eta)=\frac12\tr (\xi \eta),
\end{equation} for $\xi,\eta\in T_L  G_k(\C^n)$.  When $k=1$, this is
the \emph{Fubini-Study metric} of constant holomorphic sectional curvature equal to 4 on $\CP^{n-1}$ (see also \cite[\S 2.]{Bolton}). 

Given a  Riemann surface $M$, we interpret  any smooth map $\psi:M \to G_k(\C^n)$  as a (complex) rank-$k$ vector subbundle, also denoted by $\psi$,  of the trivial vector  bundle $M\times \C^{n}$, with fibre at $z \in M$ given by $\psi(z)$. 


Let $\psi:M\to G_k(\C^n)$ be a smooth map.
Take a local complex coordinate $z$ on $U\subset M$.  The \emph{second fundamental forms} $A'_\psi,A''_\psi\in \Hom(\psi| _U ,\psi^\perp|_U )$  (see \cite[\S 1.]{burstall-wood}) are defined
by
$$A'_\psi(s)= \pi_{\psi^\perp} \circ \frac{\partial}{\partial z} s, \quad  A''_{\psi}(s) = \pi_{\psi^\perp} \circ \frac{\partial}{\partial \bar z} s,$$
where $s$ is a section of $\psi|_U$ and  $\pi_{\psi^\perp}$ is the orthogonal (Hermitian) projection onto $\psi^\perp$.  We clearly have $A''_{\psi}=-(A'_{\psi^\perp})^*$.
These objects can be turned into global objects by considering the vector bundle homomorphisms 
$\mathcal{A}'_\psi\in \Hom(T^{1,0}M\otimes \psi,\psi^\perp)$ and $\mathcal{A}''_\psi\in \Hom(T^{0,1}M\otimes \psi,\psi^\perp)$ given  locally by ${A}'_\psi\otimes dz$ and ${A}''_\psi\otimes d\bar z$, respectively. It is well known \cite[Lemma 1.3]{burstall-wood} that
\emph{$\psi$ is harmonic if and only if $\mathcal{A}'_{\psi}$ is holomorphic, and this holds if and only if $\mathcal{A}''_{\psi}$ is antiholomorphic} (with respect to the Koszul-Malgrange complex structures of $\psi$ and $\psi^\perp$).

Given a harmonic map $\psi:M\to G_k(\C^n)$, the holomorphicity of $\mathcal{A}'_{\psi}$  is used  (see \cite[\S 2.]{burstall-wood}) to define the   first \emph{$\partial'$-Gauss bundle} $\psi_1=\underline{\Img} \,\mathcal{A}_\psi'$ of $\psi$ as the unique vector subbundle of $\psi^\perp$ that coincides with the image of $\mathcal{A}_\psi'$ almost everywhere.  If $\psi_1=\{0\}$, then $\psi$ is antiholomorphic; otherwise $\psi_1$ is a harmonic subbundle of rank $ k'\leq k$. Similarly,   the antiholomorphicity of $\mathcal{A}''_{\psi}$  is used to define the first \emph{$\partial''$-Gauss bundle} $\psi_{-1}=\underline{\Img} \,\mathcal{A}_\psi''$ of $\psi$.  If $\psi_{-1}=\{0\}$, then $\psi$ is holomorphic; otherwise $\psi_{-1}$ is a harmonic subbundle of rank  $k''\leq k$.  Hence, proceeding recursively, we can associate to a harmonic map $\psi:M\to G_k(\C^n)$ a sequence
$  \{\psi_j\}_{j\in\mathbb Z} $
of harmonic subbundles, with $\psi=\psi_0$, which is called the associated \emph{harmonic sequence} (see \cite[\S 2. and \S 3.]{burstall-wood}).

If $\psi$ is holomorphic, then $\psi$ is  \emph{strongly isotropic} \cite[(1.7)]{erdem-wood} (i.e. it has \emph{isotropy order} $\geq r$ for all $r$); it follows from \cite[Lemma 3.1]{burstall-wood} that the harmonic sequence of $\psi$ must take the form
$$\{0\}=\psi_{-1},\psi_0,\psi_1,\dots,\psi_p,\psi_{p+1}=\{0\},\quad \mbox{with $\psi_i\perp \psi_j$ for all $i\neq j\in\{0,\ldots, p\}$}.$$
A harmonic map from a Riemann surface into a complex Grassmannian  is called \emph{pseudoholomorphic} if it belongs to the harmonic sequence of some holomorphic map.  In particular, all pseudoholomorphic maps are strongly isotropic.  

\begin{remark}
	\begin{enumerate}
	\item[a)] Using the terminology of loop group theory for harmonic maps from Riemann surfaces into symmetric spaces (first introduced by K. Uhlenbeck \cite{uhlenbeck}), all pseudoholomorphic maps are harmonic maps of  \emph{finite uniton number} \cite[Proposition 4.9]{APW1}.  General criteria for finiteness of the uniton number for harmonic maps into complex Grassmannians in terms of the corresponding harmonic sequences and diagrams have been recently developed in  \cite{APW1,PW1}. 
\item[b)] The nonexistence of nonzero holomorphic differentials on the two-sphere $S^2$  ensures that all harmonic maps  from $ S^2$ into $\CP^{n-1}$ are pseudoholomorphic (see \cite[Theorem 2.6]{ burstall-wood}   and \cite[\S 7.(C)]{eellswood}).
\item[c)] An example of a harmonic map from the two-torus $T^2$ into a complex projective space which is not pseudoholomorphic is provided  by the \emph{Clifford solution} $\psi:T^2\to \CP^3$, which is  defined in homogeneous coordinates by
$$\psi(z)= [e^{2i(z+\bar z)},ie^{-2i(z+\bar z)},e^{2(\bar z- z)},e^{2( z- \bar z)}].$$ 
	\end{enumerate}
\end{remark}

 Let $\psi_0:M\to G_k(\C^n)$ be a holomorphic map and $\psi_0,\psi_1,\dots,\psi_p$ be its harmonic sequence.  
  In view of \eqref{eq:complexstructure}, we have  (see \cite[\S 1.C]{burstall-wood})
 $$d\psi_j\big(\tfrac{\partial}{\partial z}\big)^{1,0}=A'_{\psi_j},\quad d\psi_j\big(\tfrac{\partial}{\partial z}\big)^{0,1}=-A'_{\psi_{j-1}},$$
  for each $j\in\{0,\ldots, p\}$. Set 
 \begin{equation}\label{gammaj}
 \gamma_j=\tr  A'_{\psi_j}(A'_{\psi_j})^*.
 \end{equation} 
 Each $\psi_j$ is  a branched  conformal immersion and the  metric induced from \eqref{hmetric} by  $\psi_j$ on $M$  is  locally given by 
 $\psi_j^*h=(\gamma_{j-1}+\gamma_j) dzd\bar z;$
  the corresponding curvature is given by \cite[\S 3.]{JiaoPeng2003}
\begin{equation}\label{curvaturepsi}
K(\psi_j)=-\frac{2}{\gamma_{j-1}+\gamma_j} \frac{\partial^2}{\partial z\partial \bar z}\log (\gamma_{j-1}+\gamma_j).
\end{equation}
Denoting by $\theta_j$ the \emph{K\"ahler angle} of $\psi_j$, we have \cite[\S 3.]{JiaoPeng2003}
\begin{equation}\label{kahler}
(\tan (\theta_j/2))^2=\frac{\gamma_{j-1}}{\gamma_j}.
\end{equation}

\begin{remark}\label{pluckerembedding}
	Consider the \emph{Pl\"{u}cker embedding}  $\iota:G_{k}(\C^n) \hookrightarrow \CP^{{n \choose k}-1}.$
	This is a holomorphic isometry. Hence, if 
	$\psi:M \to G_{k}(\C^n)$ is a holomorphic map, and $s$ is a local nonvanishing section of $\iota\circ \psi$,   then the metric   on $M$ induced by $\psi$  from $h=ds^2_{ G_{k}(\C^n)}$  is locally  given by (see \cite{Bolton, Chi-Zheng,Jiao2008,JiaoPeng2003})
	\begin{equation}\label{metricgrasshol}
		\psi^*ds^2_{ G_{k}(\C^n)}=  \frac{\partial^2}{\partial z\partial \bar z}\log \|s\|^2dzd\bar z.
	\end{equation}
\end{remark}

\subsection{The Veronese sequence}
Recall that the $n$-\emph{Veronese map} $V^n:S^2\to \C P^n$ is the holomorphic map given by
$$V^n(z) =\left[1, \sqrt{n \choose 1}  z, \ldots, \sqrt{n \choose r} z^r, \ldots, z^n\right],
	$$in terms of the complex coordinate $z$ on $\C\subset S^2.$
	This is a linearly full hollomorphic immersion of constant curvature. Moreover, up to unitary congruence, the $n$-Veronese map is the unique such immersion, as shown by E. Calabi (see \cite{LawsonCalabi} and the references therein).  The harmonic sequence $V_0^n, \ldots, V_n^n$  associated to $V^n=V^n_0$ is called the \emph{$n$-Veronese sequence}.
	\begin{theorem}\label{veronesebolton} \cite{Bolton}
		The Veronese sequence $V_0^n,\ldots, V^n_n :S^2 \rightarrow \mathbb{C}P^n$ is given by
		$$V_j^n= \left[f_{j,0},\ldots, f_{j,n} \right]$$where,
		$$f_{j,r}(z)= \frac{j!}{(1+z\bar{z})^j}  \sqrt{n \choose r}  z^{r-j} \sum_{k} (-1)^k   {r \choose j-k}  {n-r \choose k} 
		(z\bar{z})^k.$$
		Moreover, $V_j^n$ is a minimal immersion with induced metric 
		$${V_j^n}^*h=(\gamma_{j-1}+\gamma_j)dzd\bar z=\frac{n+2j(n-j)}{(1+z\bar{z})^2}dzd\bar{z},\quad \gamma_j=\frac{(j+1)(n-j)}{(1+z\bar z)^2}$$ and having  constant curvature $K(V_j^n)= \frac{4}{n+2j(n-j)}$.
	\end{theorem}

	\subsection{Primitive harmonic immersions } \label{sec:intro}
	We consider the flag manifold $F_{k_0,\ldots,k_p }$ consisting of all $(p+1)$-tuples  $\Psi=(\psi_0,\dots, \psi_p)$ of mutually orthogonal complex subspaces of $\C^{n}$, so that $\C^n=\psi_0\oplus\ldots \oplus \psi_p$ and $k_j=\dim \psi_j$, for each $j\in\{0,\dots,p\}$. 
	As a homogeneous space, this flag manifold is given by 
	$$F_{k_0,\ldots,k_p }=\frac{\U(n)}{\U(k_0)\times\ldots \times \U(k_p)},$$
	where $\U(n)$ acts on $F_{k_0,\ldots,k_p }$ as  $g\Psi=(g\psi_0,\dots g\psi_p)$ for  $g\in \U(n)$.  
	Moreover, we have the following identification for the tangent space of the flag  manifold at $\Psi$:
	\begin{equation}\label{eq:TFdecomposition}
	T^\C_\Psi F_{k_0,\ldots,k_p }\cong \sum_{i\neq j}\Hom(\psi_i, \psi_{j}).
	\end{equation}
	We distinguish the subbundle $T^1$ of $T^\C F_{k_0,\ldots,k_p }$ whose fiber at $\Psi$ is given by 
	$$T^1_\Psi\cong \sum_{j\in\mathbb{Z}_{p+1}} \Hom(\psi_j, \psi_{j+1}).$$

		\begin{definition}\cite[\S 1.3]{burstall}
		Let $M$ be a Riemann surface. A smooth map $\Psi:M\to F_{k_0,\ldots,k_p }$, with $\Psi=(\psi_0,\ldots,\psi_p)$  is said to be \emph{primitive} if $d\Psi \big(\tfrac{\partial}{\partial z}  \big)$ is a local section of $\Psi^*T^1$, for all local complex coordinate $z$ on $M$. 
\end{definition}

\begin{remark}
	\begin{enumerate}\label{remarkprimitive}
	\item The term ``primitive" was  introduced \cite{burstall} in the more general setting of maps from surfaces into $k$-symmetric spaces. 
	(The  flag manifold $F_{k_0,\ldots,k_p }$ carries a structure of $k$-symmetric space, with $k=p+1$.) See also Chapter 21 of \cite{Guest}.
	\item For $p>1$, if $\Psi:M\to F_{k_0,\ldots,k_p }$ is primitive, then each $\psi_j:M\to G_{k_j}(\C^n)$ is harmonic \cite[Proposition 1.3]{burstall}. (For $p=1$, all maps are primitive.) 
	\end{enumerate}
\end{remark} 
 A result due to M. Black \cite{black} (see  \cite[Proposition 1.1]{burstall}) implies  that, for $p>1$, any primitive map into $F_{k_0,\ldots,k_p }$ is harmonic with respect to all $\U(n)$-invariant metrics of $F_{k_0,\ldots,k_p }$. Any such metric $g$ has the following form \cite[p.104]{Andreas Arvanitoyeorgos}: given $\xi,\eta\in T_\Psi F_{k_0,\ldots,k_p }$, and writing $\xi=\sum\xi_{ij}$, $\eta=\sum\eta_{ij}$ according to \eqref{eq:TFdecomposition}, then 
\begin{equation}\label{invariantmetric}
g(\xi,\eta)=\sum_{i\neq j} \lambda_{ij}\tr \xi_{ij}\eta_{ij}^*,
\end{equation}
for some positive constants $\lambda_{ij}$ satisfying $\lambda_{ij}=\lambda_{ji}$.  Hence, while the Grassmannian admits a unique (up to multiplication by a positive constant) $\U(n)$-invariant metric, the flag $F_{k_0,\ldots,k_p }$ admits infinitely many nonequivalent $\U(n)$-invariant metrics.

	 Given a linearly full holomorphic harmonic map $\psi_0:M\to G_{k_0}(\C^n)$, with harmonic sequence $\psi_0,\ldots, \psi_p$, we can define $\Psi:M\to F_{k_0,\ldots,k_p }$ by setting $\Psi=(\psi_0,\dots,\psi_p)$, where $k_j$ is the dimension of the fibers of $\psi_j$.  It follows directly from the definition of harmonic sequence that  $\Psi$ is primitive. If $\Psi$ is an immersion, then the metric on $U\subset M$ induced by $\Psi$ from a metric on $ F_{k_0,\ldots,k_p }$ of the form  \eqref{invariantmetric} is given by
	 \begin{equation}\label{eq:metric1}
	 \Psi^*g=\sum_{j=0}^{p-1} \lambda_j \gamma_j dzd\bar z,
	 \end{equation}
	 where we denote, for each $j\in\{0,\ldots, p-1\}$,  $\lambda_j=\lambda_{j,j+1}$; $\gamma_j$ is given by \eqref{gammaj}. The corresponding curvature is given by
	 \begin{equation}\label{kpsi}
	 K(\Psi)=-\frac{2}{\sum_{j=0}^{p-1}  \lambda_j \gamma_j } \frac{\partial^2}{\partial z\partial \bar z}\log \sum_{j=0}^{p-1}  \lambda_j \gamma_j .
	 \end{equation}
	 \begin{eg}
	 We consider the $n$-Veronese sequence $V^n_0,\ldots, V_n^n$ and the corresponding primitive map  $\mathcal{V}=(V^n_0,\ldots, V_n^n ):S^2 \to F_{1,\ldots,1}$, which will be called the $n$-\emph{Veronese primitive map}. In view of Theorem \ref{veronesebolton},  given an invariant metric on $F_{1,\ldots,1}$ with positive constants $\lambda_0,\ldots, \lambda_{n-1}$, according to \eqref{eq:metric1}, the induced metric on $S^2$ from the flag is given by
	 $$\mathcal{V}^*g= \sum_{j=0}^{n-1}  \lambda_{j}\frac{(j+1)(n-j)}{(1+z\bar{z})^2}dzd\bar z,$$
which has constant curvature
$$K(\mathcal{V})= \frac{4}{\sum_{j=0}^{n-1}  \lambda_{j}(j+1)(n-j)}.$$
 \end{eg}

	\section{Primitive harmonic immersions of constant curvature} 
	\begin{theorem}\label{principaltheorem} Let $\psi_0:S^2 \to G_{k_0}(\C^n)$  be a linearly full holomorhic map, with harmonic sequence $\psi_0,\ldots, \psi_p$. 
	If the primitive lift $\Psi=(\psi_0,\ldots, \psi_p):S^2\to F_{k_0,\ldots,k_p }$ is an immersion and there exists at least one $\U(n)$-invariant metric on $F_{k_0,\ldots,k_p }$ with respect to which  $\Psi$ has constant curvature, then  $\Psi$ has constant curvature with respect to all $\U(n)$-invariant metrics on $F_{k_0,\ldots,k_p }$; moreover, each $\psi_j:S^2\to G_{k_j}(\C^n)$ is a minimal immersion of constant curvature and constant K\"{a}hler angle.  
\end{theorem}

		\begin{proof}  
			Let $g$ be  a $\U(n)$-invariant metric on $F_{k_0,\ldots,k_p }$ with respect to which  $\Psi$ has constant curvature. 
			By Gauss-Bonnet theorem, the constant curvature must be positive; hence, by Minding theorem, we have 
			\begin{equation}\label{metric2}
				\Psi^*g=\frac{\alpha}{(1+z\bar z)^2}dzd\bar z,
			\end{equation}
			for some constant $\alpha>0$ and local complex coordinate $z$ on $\C\subset S^2$. On the other hand, by \eqref{eq:metric1},
			$\Psi^*g=\sum_{j=0}^{p-1} \lambda_j \gamma_jdzd\bar z$  for some positive constants $\lambda_0,\ldots,\lambda_{p-1}$.

		Set $k^{(j)}=k_0+\ldots+k_j$ and consider the holomorphic map $$\psi^{(j)}=\psi_0\oplus\ldots\oplus \psi_j:S^2\to G_{k^{(j)}}(\C^n).$$   Following \cite{Bolton,Chi-Zheng,Jiao2008,JiaoPeng2003}, we take the composition of $\psi^{(j)}$ with the Pl\"{u}cker embedding  
			 $\iota: G_{k^{(j)}}(\C^n)\to \CP^{{n \choose k^{(j)}}-1}$  to obtain  a holomorphic map 
			 $\sigma_j:S^2\to\CP^{{n \choose k^{(j)}}-1}$. This holomorphic map $\sigma_j$ satisfies
			  ${\sigma_j}^*h_j=\gamma_jdzd\bar z,$ where $h_j$ stands for the Fubini-Study metric on $\CP^{{n \choose k^{(j)}}-1}$.
			 Let $\hat \sigma_j:\C\to \C^{{n \choose k^{(j)}}}$ be a holomorphic local section of $\sigma_j$. Without loss of generality, we can assume that $\hat \sigma_j$ is nowhere zero, as we can always remove the greatest common divisor of its components.  Set $\beta_j=\|\hat \sigma_j\|^2,$ which is polynomial in $z$ and $\bar z$. By \eqref{metricgrasshol}, we have
		\begin{equation}\label{eq:logbetaj}
			\frac{\partial^2 }{\partial z \partial \bar z} \log \beta_j=\gamma_j,
		\end{equation}    
		hence 
		\begin{equation}\label{metric3}
			\Psi^*g=\frac{\partial^2 }{\partial z \partial \bar z} \log\left(\beta_0^{\lambda_0}\ldots  \beta_{p-1}^{\lambda_{p-1}}\right) dzd\bar z
		\end{equation}
		Combining \eqref{metric2} and \eqref{metric3}, we deduce that the following holds on $\C$:
		\begin{equation*}
			\frac{\partial ^2}{\partial{\bar z} \partial z} \log \frac{ \beta_{0}^{\lambda_{0}}\ldots \beta_{p-1}^{\lambda_{p-1}}}{(1+z \bar{z})^\alpha}=0.
		\end{equation*}
		So, $\log \frac{ \beta_{0}^{\lambda_{0}}\ldots \beta_{p-1}^{\lambda_{p-1}}}{(1+z \bar{z})^\alpha}$ is a harmonic function on $\mathbb{C},$ which implies that there exists an entire function $f$ such that 
		\begin{equation}\label{eq:betah}
			\log \frac{ \beta_{0}^{\lambda_{0}}\ldots \beta_{p-1}^{\lambda_{p-1}}}{(1+z \bar{z})^\alpha}=f+\bar{f}.
		\end{equation}
		Exponentiate \eqref{eq:betah} to obtain 
		\begin{equation}\label{eq:betah1}
		 \beta_{0}^{\lambda_{0}}\ldots \beta_{p-1}^{\lambda_{p-1}}= (1+z\bar z)^\alpha e^{f}\overline{e^{f}}.
		 \end{equation}
			Since all polynomials $\beta_j(z,\bar z)$, with $j\in\{0,\ldots,p-1\}$, and the factor $1+z\bar z$ are nonvanishing, the functions
		$z\mapsto \beta_j^{\lambda_j}(z,\bar z)$ and $z\mapsto  (1+z\bar z)^\alpha$
		are real analytic on $\mathbb{C}$, even when $\lambda_j$ and $\alpha$ are not integers. The corresponding complexifications are given, respectively, by 
		$(z,w)\mapsto \beta_j^{\lambda_j}(z,w)$ and $(z,w)\mapsto  (1+zw)^\alpha,$
		considering the principal branch of each multivalued function $Z\mapsto Z^q$. 
		Now, take the complexifications of both sides of  \eqref{eq:betah1};
	by uniqueness of complexification,
		we have 	
		\begin{equation*}\label{eq:betah2}
			\beta_{0}^{\lambda_{0}}(z,w)\ldots \beta_{p-1}^{\lambda_{p-1}}(z,w)=(1+zw)^\alpha e^{f(z)}e^{\bar f (w)}
		\end{equation*}
		on an open set $V\subset \mathbb{C}\times  \mathbb{C}$ containing $\{(z,w):\, w=\bar z \}$.	Hence, if not empty, the zero set of  $\beta_j(z,w)$
		coincides with the zero set of the irreducible factor  $1+zw$, which implies that	\begin{equation}\label{eq:betajCj}
			\beta_j(z,\bar z)=(1+z\bar z)^{\alpha_j}C_j
		\end{equation}
		for some real constant $C_j$ and nonnegative integer  $\alpha_j$. 
		It follows from \eqref{eq:logbetaj} and \eqref{eq:betajCj}  that
		$$ \gamma_{j}= \frac{\alpha_j}{(1+z \bar{z})^2}, \quad j\in\{0,\ldots,p-1\}.$$
		If $\alpha_j=0$ for some $j\in\{0,\ldots,p-1\},$ then $\gamma_j=0$, which implies that the holomorphic map $\sigma_j$ is constant, giving a contradiction with the hypothesis of $\psi_0$ being full. Hence, $\alpha_j>0$ for all $j\in\{0,\ldots,p-1\}$. Consequently,   taking account of  \eqref{curvaturepsi} and \eqref{kahler}, we see that all $\psi_j$  are   immersions of constant curvature and constant K\"{a}hler angle.  
		
	If $\tilde g$ is any other $\U(n)$-invariant metric on $F_{k_0,\ldots,k_p}$, then,  for some positive constants $\tilde \lambda_0,\ldots,\tilde \lambda_{p-1}$, we have 
		$$\Psi^*\tilde g= \sum_{j=0}^{n-1} \tilde \lambda_j \gamma_jdzd\bar z= \frac{\sum_{j=0}^{n-1}\tilde \lambda_j \alpha_j}{(1+z \bar{z})^2}dzd\bar z.$$ Using \eqref{kpsi} to compute the curvature, it follows that $\Psi$ has constant curvature
		$\frac{4}{\sum_{j=0}^{n-1}\tilde \lambda_j \alpha_j}$ with respect to $\tilde g$
	\end{proof}
	
We say that a primitive map $\Psi=(\psi_1,\ldots,\psi_p)$  is \emph{full} if all $\psi_j$ are linearly full. 

		\begin{corollary}\label{principalcorollary} Let $\Psi:S^2 \rightarrow F_{1,\ldots,1}$  be a full primitive immersion.
	If there exists at least one $\U(n)$-invariant metric on $F_{1,\ldots,1}$ with respect to which  $\Psi$ has constant curvature, then  $\Psi$ is  unitarily congruent with the $(n-1)$-Veronese primitive map.
	\end{corollary}  
	\begin{proof}
Since $\Psi=(\psi_0,\ldots, \psi_{n-1})$ is a full primitive map, each $\psi_j$, with $j\in\{0,\ldots,n-1\}$, is a full harmonic map into $\CP^{n-1}$ (see Remark \ref{remarkprimitive}),  and the first Gauss bundle of $\psi_j$ is either $\{0\}$ or $\psi_{j+1}$. 
Since all harmonic maps  from $S^2$ into $\CP^{n-1}$ are pseudoholomorphic, we can reorder (this corresponds to a unitarily congruence on the full flag) if necessary, so that $\psi_0:S^2\to \C P^{n-1}$ is holomorphic and  $\psi_0,\ldots,\psi_n$ is the corresponding harmonic sequence. By Theorem \ref{principaltheorem}, $\psi_0$ is an immersion of constant curvature, hence, by Calabi's result, $\psi_0$ is unitarily congruent with the $(n-1)$-Veronese map. 
	\end{proof}

	\begin{example}
	Let $\psi^a_0 :S^2 \rightarrow \mathbb{C}P^2 $ be the holomorphic immersion defined by 
	$$ \psi_0^a(z)= \left[1,az,z^2\right],\quad a\in \mathbb{C}\setminus\{0\}.$$ 
Observe that $\psi_0^{\sqrt 2}=V^2$. The nowhere zero local section
	$f_0(z) = (1,az,z^2)$ satisfies
	$$ \left|f_0\right|^2 =1 +|a|^2 |z|^2 + |z|^4, \quad \quad \left \langle \frac{ \partial f_0}{\partial z}, f_0 \right \rangle = \bar{z}(|a|^2+2|z|^2).$$
Set $f_1=A'_{\psi^a_0}(f_0)$.  We have
\begin{align*}
f_1(z)&=\frac{\partial f_0}{\partial z}-\frac{1}{|f_0|^2}\Big< \frac{\partial f_0}{\partial z}, f_0\Big> f_0\\=& \frac{1}{1+|a|^2|z|^2+ |z|^4}\left(-\bar{z}(|a|^2+2|z|^2),a(1-|z|^4),z(2+|a|^2|z|^2)\right).
\end{align*}
This is a nowhere zero local section of $\psi_1^a:S^2 \rightarrow \mathbb{C}P^2 $. Since
	$$\left|f_1\right|^2= \frac{4| z|^2+|a|^2(1+|z|^4)}{1 +|a|^2 |z|^2 + |z|^4},$$
	we get 
	$$ \gamma_0 = \frac{\left| f_1 \right| ^2}{\left|f_0 \right|^2} = \frac{4| z|^2+|a|^2(1+|z|^4)}{(1 +|a|^2 |z|^2 + |z|^4)^2}. $$
	Similarly,  the local section $f_2= A'_{\psi^a_1}(f_1)$ of $\psi^a_2$ is given by
	$$f_2(z)=\frac{1}{4| z|^2+|a|^2(1+|z|^4)} \left({2|a|^2 | z|^2}, {4a\bar{z}}, {2|a|^2}\right),$$
and
	$$\gamma_1= \frac{\left|f_2\right|^2}{\left|f_1\right|^2} =\frac{4|a|^2(1+|a|^2|z|^2+|z|^4)}{(4|z|^2+|a|^2(1+|z|^4))^2}.$$
	Let $\Psi^a = (\psi^a_0, \psi^a_1, \psi^a_2): S^2 \rightarrow F_{1,1,1}$ be the corresponding primitive immersion. The curvature on $S^2$  induced by $\Psi^a$ from an invariant metric on $F_{1,1,1}$ with parameters  $\lambda_0,\lambda_1$  is given by
	\begin{equation}\label{Klambda}
	K_{\lambda_0,\lambda_1}(\Psi^a)=-\frac{2}{\lambda_{0}\gamma_0+ \lambda_1 \gamma_1} \frac{\partial ^2}{\partial z \partial{\bar{z}}}\log(\lambda_{0}\gamma_0+ \lambda_1 \gamma_1). 
	\end{equation}
In Figure \ref{graph1},	we plot the graph (using \emph{Mathematica} software) of  	$K_{1,1}(\Psi^a)$ 
	and  $K_{1,2}(\Psi^a)$ for different values of $a$, as a function of the  latitude angle $\varphi$ on $S^2$, that is, taking $|z|=\cot \frac{\varphi}{2}$, with $\varphi\in[0,\pi]$. 
	\begin{figure}[h!]
	\begin{center}
		\includegraphics[width=8.25cm]{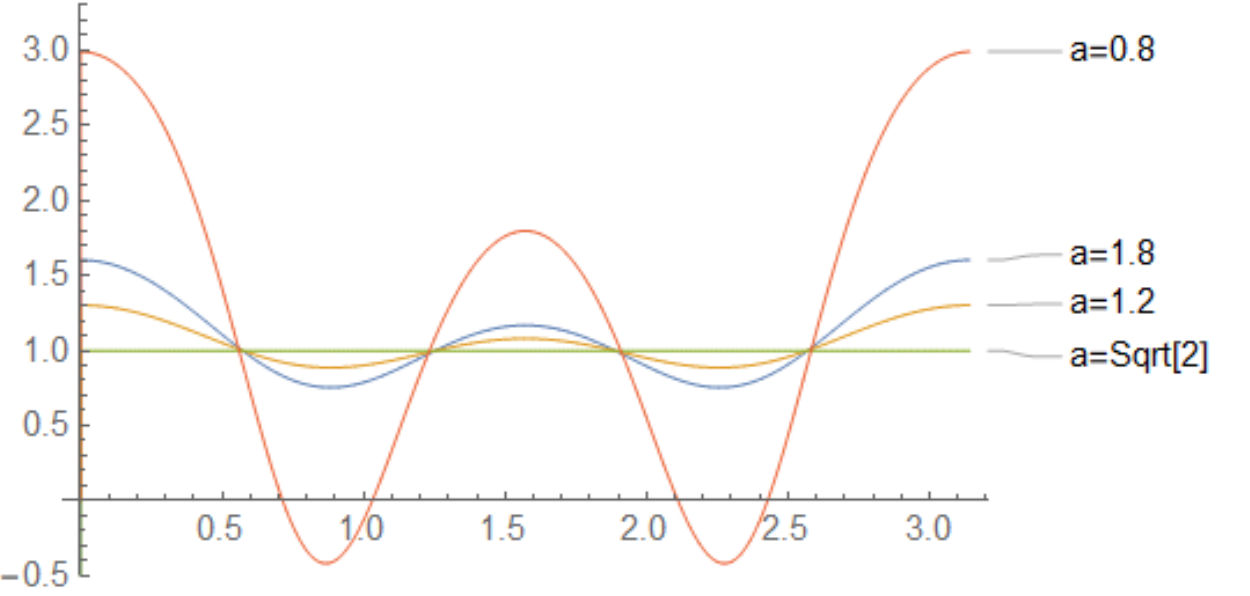}\includegraphics[width=8.25cm]{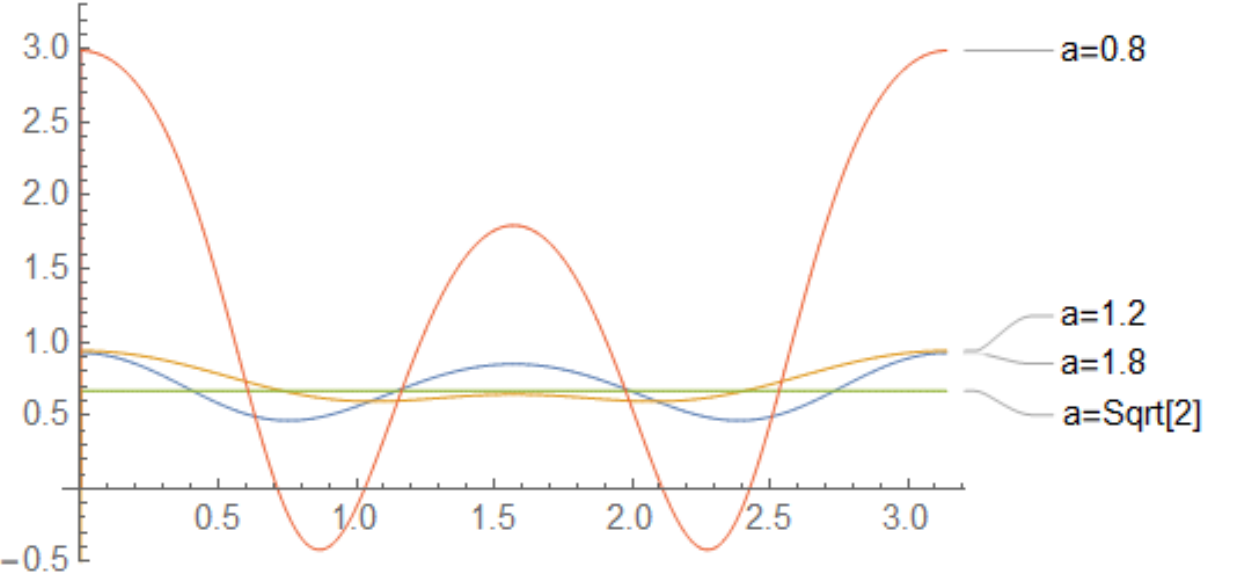}
	\end{center}
	\caption{Graphs of $K_{1,1}(\Psi^a)$ (left) and  $K_{1,2}(\Psi^a)$ (right).}
	\label{graph1}
\end{figure}
\end{example}

\section{Which invariant metric maximizes the area?}
Let $\psi_0:S^2 \to G_{k_0}(\C^n)$  be
a linearly full holomorhic map, with harmonic sequence $\psi_0,\ldots, \psi_p$. 
Suppose that the primitive lift $\Psi=(\psi_0,\ldots, \psi_p):S^2\to F_{k_0,\ldots,k_p }$ is an immersion. For a given choice of a $\U(n)$-invariant metric on the flag manifold, we have seen  the induced metric on $S^2$ takes locally the form
$\Psi^*g=\sum_{j=0}^{p-1} \lambda_j \gamma_j dzd\bar z.$
The area of $S^2$ with respect to this metric is given by 
\begin{equation}\label{Areainv}
	A(\Psi)=\frac{1}{2i}\sum_{j=0}^{p-1} \lambda_j \int_{S^2}\gamma_j d\bar z \wedge dz=\pi 
	\sum_{j=0}^{p-1} \lambda_j\delta_j, 
\end{equation}
where $$ \delta_j=\frac{1}{2\pi i}\int_{S^2}\gamma_j d\bar z \wedge dz$$ is the degree (see \cite[\S 3.]{Bolton} and \cite[Chap.1]{Griffiths}) of the 
holomorphic map $\iota \circ \psi^{j}: S^2\to \CP^{{n \choose k^{(j)}}-1}$, being $\psi^{j}=\psi_0\oplus\ldots \oplus \psi_j$ (with rank $k^{(j)}$) and $\iota$ the Plücker embedding of $G_{k^{(j)}}(\C^n)$. 
Under the normalization 
$\lambda_0^2+\ldots +\lambda_{p-1}^2=1,$
we can now easily address the problem of finding the invariant metric on the flag manifold that maximizes the induced area of $S^2$.  In view of \eqref{Areainv}, such invariant metric has parameters
$$\lambda_0=\frac{\delta_0}{\sqrt{\sum_{j=0}^{p-1} \delta_j^2}},\ldots,\lambda_{p-1}=\frac{\delta_{p-1}}{\sqrt{\sum_{j=0}^{p-1} \delta_j^2}}.$$

\begin{example}
		Consider a full holomorphic map $\psi_0:S^2\to \mathbb{C}P^2$, \emph{totally unramified}, with primitive lift
		$\Psi=(\psi_0,\psi_1,\psi_2):S^2\to F_{1,1,1}.$
		Since $\psi_{0}$ is totally unramified, it follows from \cite[Equation 3.25]{Bolton} that the degrees $ \delta_{i}$ are given  by
		$\delta_0=\delta_1=2$. 
		So, the normalized invariant metric which maximizes the area of $\Psi$ has parameters
		$\lambda_0=\lambda_1=\frac{1}{\sqrt 2}.$
\end{example}

\begin{example}
Let $\psi_0:S^2 \to G_{2}(\C^5)$  be the constant curved holomorphic immersion
		locally spanned by 
		$$f_0=\left(1, 0, 2z, 2z^2, z^2\right),\qquad g_0=\left(0, 1, 0, z^2,0\right).$$ 
	Straighforward computations show that 
		\begin{equation*}
		\gamma_{0}=\frac{\partial^2}{\partial z\partial{\bar z}} \log\left\|{f_0} \wedge g_0\right\|^2=\frac{\partial^2}{\partial z\partial{\bar z}} \log(1+z\bar{z})^4=\frac{4}{(1+z\bar{z})^2}.
	\end{equation*}
	and 
		\begin{align*}
		\gamma_1=\frac{\partial^2}{\partial z\partial{\bar z}}\log 	\left\|	g_0 \wedge f_0 \wedge\frac{\partial g_0}{\partial z }  \wedge \frac{\partial f_0}{\partial z} \right\|^2
		= \frac{1+|z|^2(4+|z|^2)}{(1+|z|^2(1+|z|^2))^2}.
	\end{align*}
	From this we see that the first Gauss bundle $\psi_1$ is a non antiholomorphic vector bundle of  rank $2$, so $\psi_1:S^2\to G_2(\mathbb{C}^5)$ and $\psi_2:S^2\to \C P^4$. Moreover, 
	$\psi_1$ is an immersion of nonconstant curvature. Hence, by  Theorem \ref{principaltheorem}, the primitive map $\Psi=(\psi_0,\psi_1,\psi_2):S^2\to F_{2,2,1}$ has nonconstant curvature with respect to all $\U(5)$-invariant metrics on $ F_{2,2,1}$.

Since $\delta_0=4$ and $\delta_1=2$, we see that, under the normalization $\lambda_0^2+\lambda_1^2=1$,  the metric  which maximizes the area of $S^2$ is given by $\Psi^*g= \left(\lambda_0\gamma_{0}+ \lambda_1\gamma_{1}\right)dzd\bar z$ with $\lambda_0=\frac{2}{\sqrt 5}$, $\lambda_1=\frac{1}{\sqrt 5}$.
		 The curvature on $S^2$ induced by $\Psi$ from an invariant metric on $F_{2,2,1}$ is again given by \eqref{Klambda}.
	Figure \ref{areagraph} shows the graphs  of the curvature $K_{\lambda_0,\lambda_1}(\Psi)$  for different values of $\lambda_0,\lambda_1$, as  functions of the  latitude angle $\varphi$ on $S^2$.
	
	\begin{figure}[h!]
		\begin{center}
			\includegraphics[width=8.25cm]{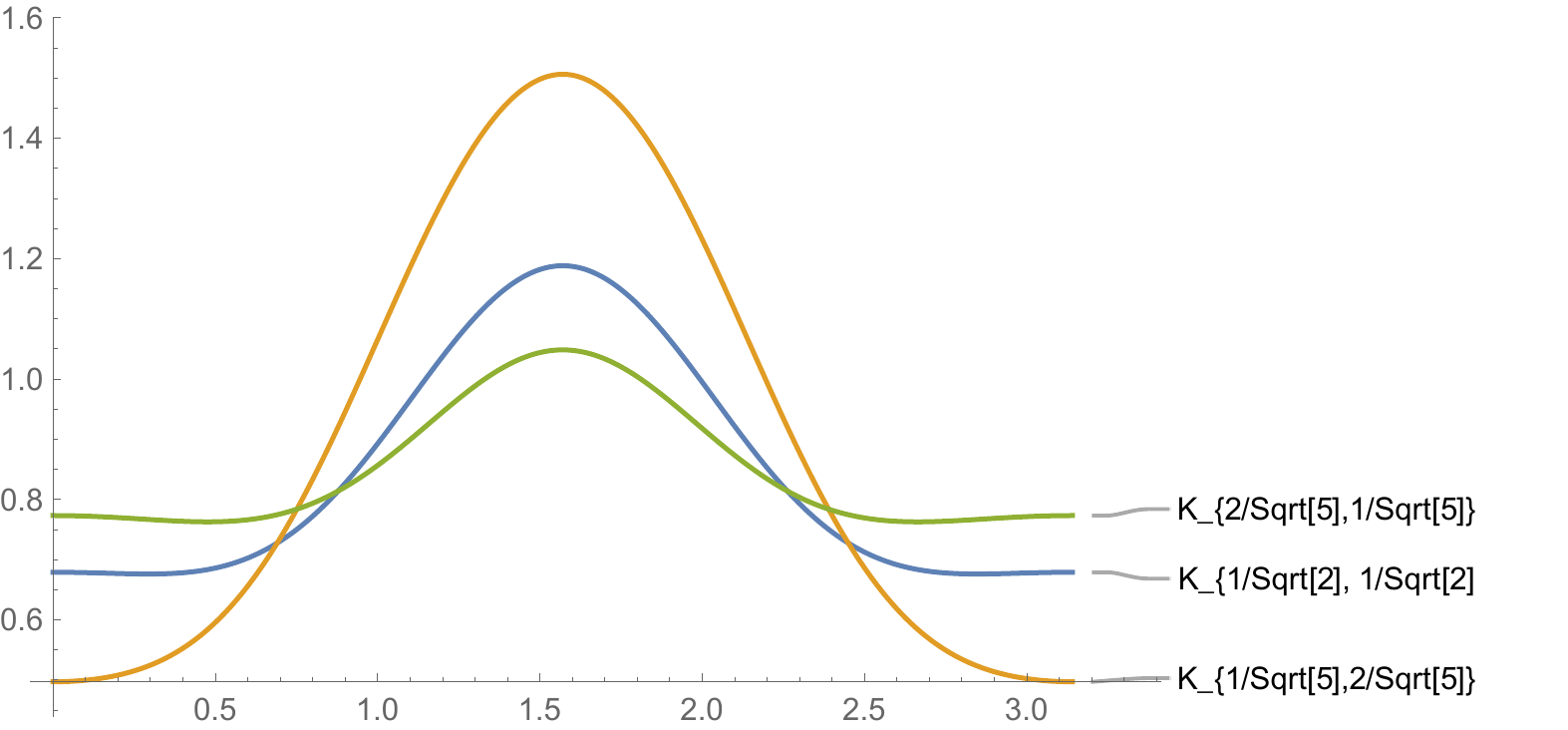}
		\end{center}
		\caption{}
		\label{areagraph}
	\end{figure}
	\end{example}

\begin{example}
In \cite{ChiXieXU}, the authors found an interesting example of a constant curved	holomorphic immersion $\psi_0:S^2 \to G_{2}(\C^5)$  of degree $\delta_0=6$: this is   
	locally spanned by 
$$f_0=\left(1, 0, -\sqrt{6}z^2, -2z^3, -3z^4\right),\qquad g_0=\left(0, 1, \sqrt 6 z, 3z^2,4z^3\right).$$ 
Straighforward computations show that 
\begin{equation*}
	\gamma_{0}=\frac{6}{(1+|z|^2)^2},\qquad \gamma_1=4\frac{1+9|z|^4+20|z^6|+9|z|^8+|z|^{12}}{(1+4|z|^2+4|z|^6+|z|^8)^2}.
\end{equation*}
From this we see that the first Gauss bundle $\psi_1$ is a non antiholomorphic vector bundle of  rank $2$, so $\psi_1:S^2\to G_2(\mathbb{C}^5)$ and $\psi_2:S^2\to \C P^4$. Moreover, 
$\psi_1$ is an immersion  of nonconstant curvature. Again, the primitive map $\Psi=(\psi_0,\psi_1,\psi_2):S^2\to F_{2,2,1}$ has nonconstant curvature with respect to all $\U(5)$-invariant metrics on $ F_{2,2,1}$.	

Since $\delta_0=6$ and by straightforward computation $\delta_1=4$, we see that, under the normalization $\lambda_0^2+\lambda_1^2=1$,  the metric which maximizes the area of $S^2$ has parameters $\lambda_0=\frac{3}{\sqrt{13}}$, $\lambda_1=\frac{2}{\sqrt{13}}$.
Figure \ref{areagraph1} shows the graphs  of the curvature $K_{\lambda_0,\lambda_1}(\Psi)$  for different values of $\lambda_0,\lambda_1$, as  functions of the  latitude angle $\varphi$ on $S^2$.  
	\begin{figure}[h!]
	\begin{center}
		\includegraphics[width=8.25cm]{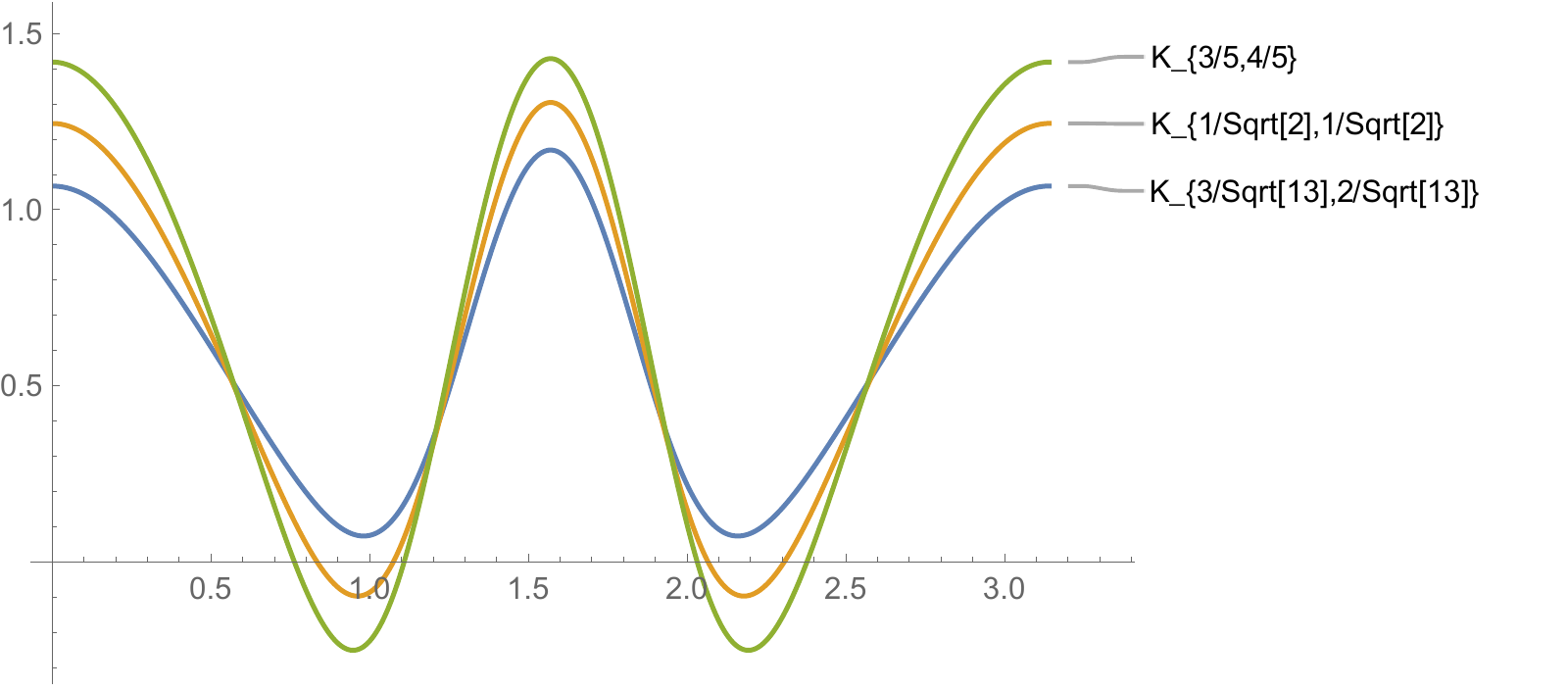}	
	\end{center}
\caption{}\label{areagraph1}
\end{figure}
	
\end{example}

\section{Primitive harmonic immersions of constant curvature of a simply connected surface.}
In this section our  aim is to generalize Corollary \ref{principalcorollary} to primitive harmonic immersions of a simply connected surface, not necessarily closed. The technique introduced by Q.-S. Chi and Y. Zheng \cite{Chi-Zheng} will play an important role in this generalization. 

Let $M$ be a simply connected surface and $\psi_0:M\to \C P^n$ be a full holomorphic curve. Let $\Psi=(\psi_0,\ldots, \psi_{n}):M\to F_{1,\ldots,1}$ be the corresponding primitive map, which is assumed  to be an immersion, and $\sigma_j:M\to\CP^{{n+1 \choose j+1}-1}$ be the $j$-th osculating curve of $\psi_0$.  On the full flag manifold $F_{1,\ldots,1}$, we fix an  invariant metric $g=ds^2_F$ such that 
\begin{equation}\label{kinteger}
\Psi^*ds^2_F=\sum_{j=0}^{n-1}k_j\gamma_jdzd\bar z, \quad \mbox{where each $k_j$ is a positive integer.}
\end{equation}
Recall that, in view of \eqref{metricgrasshol}, $\gamma_j=	\frac{\partial^2 }{\partial z \partial \bar z} \log \|\hat \sigma_{j} \|^2,$ where $\hat \sigma_j$ is a nonvanishing section of $\sigma_j$. 

 We define 
$\eta:M\to \CP^{N}$ by
$$\eta=\sigma_0^{k_0}\otimes\sigma_1^{k_1}\otimes \ldots \otimes \sigma^{k_{n-1}}_{n-1},$$
where
$$N=L_0^{k_0}\ldots L_{n-1}^{k_{n-1}}-1,\quad \mbox{with $L_j={n+1 \choose j+1}$}.$$
Here we are denoting 
$$\sigma_j^{k_j}=\underbrace{\sigma_{j}\otimes \ldots \otimes\sigma_{j}}_{\mbox{$k_j$ times}},$$
for each $j\in \{0,1,...,n-1\}$. Notice that the map $\eta$ is \emph{holomorphic}, because this map is defined as a tensor product of holomorphic osculating curves.

\begin{lemma}\label{same metric}  Under the above assumptions, 
$$\eta^* ds^2_{\CP^N} =\Psi^*ds^2_{F},$$
where  $ds^2_{\CP^N}$ stands for the Fubini-Study metric on $\CP^N$. 
\end{lemma}
\begin{proof}Let $\hat\eta:\C \to \C^{N+1}$ be a local section of  $\eta$ of the form
$$\hat\eta= \hat \sigma_0^{k_0} \otimes \hat \sigma_1^{k_1}  \otimes\ldots \otimes \hat \sigma_{n-1}^{k_{n-1}}, $$ where, for each $j$, $\hat{\sigma_{j}} : \C \to \C^{L_j}$ is a nonvanishing local section of the osculating curve $\sigma_{j}$. Taking norm squares yields
$\|\hat\eta \|^2=  \|\hat \sigma_0\|^{2k_0}\ldots \|\hat \sigma_{n-1}\|^{2k_{n-1}},$
then
\begin{align*}
	\frac{\partial^2 }{\partial z \partial \bar z} \log \|\hat\eta \|^2= k_{0}\frac{\partial^2 }{\partial z \partial \bar z}\log\|\hat\sigma_0\|^{2}+\ldots + k_{n-1}\frac{\partial^2 }{\partial z \partial \bar z}\log\|\hat \sigma_{j}\|^{2}.\end{align*}
By \eqref{metricgrasshol}, we conclude	that	
$$\eta^* ds^2_{\CP^N}= \sum_{j=0}^{n-1} k_j \gamma_{j}=\Psi^*ds^2_{F}.$$ 
\end{proof}
\begin{theorem}
Let  $M$ be a simply connected Riemann surface and $\psi_0:M\to \CP^n$ be a full holomorphic map, with $n>1$. 
If the lift $\Psi:M \rightarrow F_{1,\ldots,1}$, with $\Psi=(
\psi_0, \ldots, \psi_n)$, is  a primitive immersion of constant curvature with respect to at least one invariant metric on $F_{1,\ldots,1}$ satisfying \eqref{kinteger}, then  $\psi_0$ is locally congruent with the Veronese map.
\end{theorem}  
\begin{proof}
Assume $\Psi$ has constant curvature $K$, with induced metric of the form  \eqref{kinteger}. Thus, by Lemma \ref{same metric}, the holomorphic map $\eta$ has constant curvature $K$. Hence,  $\eta$ is locally unitarily equivalent to a portion of the Veronese map $V_{N}$. 
Take local nonvanishing holomorphic sections $\hat \sigma_j$ of $\sigma_j$, with $j\in\{0,\ldots , n-1\}$, such that $\eta$ is locally spanned by 
$$\hat\eta= \hat \sigma_0^{k_0}\otimes \hat \sigma_1^{k_1}\otimes \ldots \otimes \hat \sigma^{k_{n-1}}_{n-1} ,$$
and 
\begin{equation}\label{etaveronese}
\hat \eta=A\hat V_N 
\end{equation}
for some constant $A\in \U(N)$, where $\hat V_N$ is the standard section of the Veronese map.  
For each positive integer $m$, set $[m]= \{0,1,\ldots, m-1\}$. Denote by $(\hat\sigma_j)_i$ the $i$-th component of $\hat \sigma_j$, with $i\in [L_j]$.
Comparing components of both sides of \eqref{etaveronese} 
we get,  for each multiindex 
\begin{equation}\label{multiindex}
I=(i_0^1,i_0^2,\ldots, i_0^{k_0}, i_1^1,i_1^2,\ldots, i_1^{k_1},\ldots)\in [L_0]^{k_0}\times [L_1]^{k_1}\times \ldots\times [L_{n-1}]^{k_{n-1}},
\end{equation}
the following 
\begin{equation}\label{defPI}
(\hat \eta)_I=(\hat \sigma_0)_{i^1_0} (\hat \sigma_0)_{i^2_0}\dots  (\hat \sigma_0)_{i^{k_0}_0}  (\hat \sigma_1)_{i_1^1} (\hat \sigma_1)_{i_1^2} \ldots 
(\hat \sigma_{n-1})_{i^{k_{n-1}}_{n-1}}=P^I(z)
\end{equation}
where $P^I$ is polynomial in $z$. We assume, without loss of generality, that 
$(\hat \sigma_j)_0$ is nonvanishing for all $j$. Then
\begin{equation}\label{bigsigmafraction}
\frac{(\hat \sigma_j)_l}{(\hat \sigma_j)_0}=\frac{(\hat \sigma_0)^{k_0}_0 \ldots(\hat \sigma_{j-1})^{k_{j-1}}_0(\hat \sigma_{j})_l(\hat \sigma_{j})^{k_j-1}_0(\hat \sigma_{j+1})^{k_{j+1}}_0 \ldots (\hat\sigma_{n-1})^{k_{n-1}}_0 }{(\hat \sigma_0)^{k_0}_0 \ldots(\hat \sigma_{j-1})^{k_{j-1}}_0(\hat \sigma_j)^{k_j}_0(\hat \sigma_{j+1})^{k_{j+1}}_0 \ldots (\hat\sigma_{n-1})^{k_{n-1}}_0}=\frac{P^{I( l,j)}}{P^{I(0,j)}},
\end{equation}
where $I(l,j)$ is the multindex of the form \eqref{multiindex} with $i^1_j=l\in [L_j]$ and all the other slots equal to zero. Observe that $I(0,j)=\vec 0$.
The equality  \eqref{bigsigmafraction} implies that
\begin{equation*}
\hat\sigma_j
= \frac{(\hat\sigma_j)_0}{P^{\vec 0}}\left(P^{I(0,j)} , P^{I(1,j)}, P^{I(2,j)},\ldots ,P^{I(L_j-1,j)}   \right),
\end{equation*} 
hence 
\begin{equation}\label{tildesigmacompon}
\|\hat\sigma_j\|^2
= \frac{|(\hat\sigma_j)_0|^2}{|P^{\vec 0}|^2}\sum_{l\in[L_j]}|P^{I(l,j)}|^2.
\end{equation} 
Observe also from \eqref{defPI}
\begin{equation}\label{PO2}
|P^{\vec 0}|^2=\prod_{j\in [n]}
|(\hat \sigma_j)_0|^{2k_j} .
\end{equation}

From  \eqref{etaveronese} we have 
\begin{equation}\label{tildeetanorm1}
\|\hat \eta\|^2=\|A\hat V_N\|^2=(1+z\bar z)^N. 
\end{equation}
On the other hand, from \eqref{tildesigmacompon} and \eqref{PO2},
\begin{equation}
\label{tildeetanorm2}
\begin{aligned}
\|\hat \eta\|^2&=\|\hat\sigma_0\|^{2k_0}\ldots\|\hat\sigma_{n-1}\|^{2k_{n-1}}\\&=
\frac{1}{|P^{\vec 0}|^{2(k_0+\ldots+k_{n-1}-1)}}\prod_{j\in[n]}\Big(\sum_{l\in [L_{j}]}|P^{I(l,j)}|^2\Big)^{k_j}.
\end{aligned}
\end{equation}
Equating \eqref{tildeetanorm1} and \eqref{tildeetanorm2},  we obtain
\begin{equation}\label{factP}
(1+z\bar z)^N |P^{\vec 0}|^{2(k_0+\ldots+k_{n-1}-1)}=\prod_{j\in[n]}\Big(\sum_{l\in [L_{j}]}|P^{I(l,j)}|^2\Big)^{k_j}.
\end{equation} 
Both sides of \eqref{factP} are polynomial functions in $z$ and $\bar z$. Since $\C[z,\bar{z}]$ is a unique factorization domain and $(1+z\bar{z})$ is irreducible in $\C[z,\bar{z}]$, it follows that, for each $j\in[n]$, 
\begin{equation}\label{factP1}\sum_{l\in [L_{j}]}|P^{I(l,j)}|^2=G_j(z,\bar z)(1+z\bar z)^{N_j},
\end{equation}
for some nonnegative integer $N_j$, where $G_j$ is a polynomial function in $z$ and $\bar z$ not divisible by $1+z\bar z$.

Each $G_j$ is real since the remaining terms in equation \eqref{factP1} are real. On the other hand, observe that \eqref{factP} and \eqref{factP1} give
$$(1+z\bar z)^N |P^{\vec 0}|^{2(k_0+\ldots+k_{n-1}-1)}=(1+z\bar z)^{k_0N_0+\ldots+ k_{n-1}N_{n-1}}G_0^{k_0}(z,\bar z)\ldots G_{n-1}^{k_{n-1}}(z,\bar z).$$
Since $P^{\vec 0}$ is holomorphic (polynomial in $z$),  the factor $1+z\bar z$ does not divide $  |P^{\vec 0}|^2$, consequently
\begin{equation}\label{prodGN}
|P^{\vec 0}|^{2(k_0+\ldots+k_{n-1}-1)}=G_0^{k_0}(z,\bar z)\ldots G_{n-1}^{k_{n-1}}(z,\bar z).
\end{equation}
It follows from \eqref{prodGN}, together with  the reality of $G_j$, that $G_j(z,\bar z)= |h_j(z)|^2$ for some polynomial $h_j(z)$ in $z$.
Hence  
\begin{equation}\label{factP2}\sum_{l\in [L_{j}]}|P^{I(l,j)}|^2=|h_j(z)|^2(1+z\bar z)^{N_j}.
\end{equation}
From \eqref{tildesigmacompon} and \eqref{factP2} we have
\begin{equation*}\label{eq9}
\|\hat \sigma_{j}\|^2=|H_j(z)|^2(1+z\bar z)^{N_j}
\end{equation*}
where $H_j$ is the holomorphic function defined by $$H_j=\frac{h_j(\sigma_{j})_0}{P^{\vec 0}}.$$ 
Since, for any holomorphic function $H$,  
$$	\frac{\partial^2 }{\partial z \partial \bar z} \log|H(z)|^2=0,$$
we have
\begin{equation*}\label{metrics1}  
\gamma_j=	\frac{\partial^2 }{\partial z \partial \bar z} \log \|\hat \sigma_{j} \|^2 = \frac{N_j}{(1+|z|^2)^2}.
\end{equation*}
Taking $j=0$, this shows that $\psi_0$ has constant curvature, then $\psi_0$  is locally congruent with a Veronese map.
\end{proof}

\subsection*{Data Availability Statement} Data sharing is not applicable to this article as no new data were created or analyzed in this study.

\subsection*{Declarations}
\subsection*{Conflict of interest}
 The authors declare no conflict of interest.

	\end{document}